\documentclass[times,twoside,a4,11pt]{amsart}

\usepackage[english]{babel}
\usepackage{url}
\usepackage{amsmath}
\usepackage{graphicx}
\usepackage[colorinlistoftodos]{todonotes}
\usepackage{amsmath,amscd,amsthm}
\usepackage{amstext, amsfonts, a4}
\usepackage{amssymb}
\usepackage{tikz-cd}
\usepackage{fancyhdr}
\usepackage{enumerate}
\usepackage{cleveref}
\usepackage{xcolor}

\numberwithin{equation}{section}
\newtheorem{thm}[equation]{Theorem}

\newtheorem{prop}[equation]{Proposition}
\newtheorem{cor}[equation]{Corollary}
\newtheorem{lem}[equation]{Lemma}

\theoremstyle{definition}
\newtheorem{exmp}[equation]{Example}

\newtheorem{rem}[equation]{Remark}

\theoremstyle{plain}

\renewcommand{\phi}{\varphi}
\renewcommand{\dim}{\operatorname{\mathsf{dim}}}

\newcommand\cchar{\operatorname{\mathsf{char}}}

\newcommand\N{\operatorname{\mathsf{N}}}

\newcommand{\llangle}{\langle\!\langle}

\renewcommand{\leq}{\leqslant}
\renewcommand{\geq}{\geqslant}

\begin{document}
\title[Essential dimension of sequences of quadratic Pfister forms]{Essential dimension of sequences of quadratic Pfister forms}
	
\date{October 25, 2025}

\author{Fatma Kader B\.{i}ng\"{o}l}
\author{Adam Chapman}
\author{Ahmed Laghribi}

\address{Scuola Normale Superiore, Piazza dei Cavalieri 7, 56126 Pisa, Italia}
\email{fatmakader.bingol@sns.it}

\address{School of Computer Science, Academic College of Tel-Aviv-Yaffo, Rabenu Yeruham St., P.O.B 8401, Yaffo, 6818211, Israel}
\email{adam1chapman@yahoo.com}

\address{Univ. Artois, UR 2462, Laboratoire de Math{\'e}matiques de Lens (LML), F-62300 Lens, France}
\email{ahmed.laghribi@univ-artois.fr}

\begin{abstract}
    We study the essential dimension of the set of isometry classes of $m$-tuples $(\varphi_1,...,\varphi_m)$ of quadratic $n$-fold Pfister forms over a field $F$ such that the Witt class of $\varphi_1 \perp \ldots \perp \varphi_m$ lies in $I_q^{n+1}F$.
    We show that the essential dimension is equal to $n+1$, when $m=3$, and is either $4$ or $5$, when $n=2=\cchar k$ and $m=4$.
		
    \medskip\noindent
    {\sc Classification} (MSC 2020): 11E04, 16K20
		
    \medskip\noindent
    {\sc{Keywords:}} essential dimension, quadratic form, quaternion algebra, characteristic $2$
\end{abstract}
	
\maketitle

\section{Introduction}
Throughout this short note, let $k$ be fixed field. Let $m, n\geq 1$ be integers. We denote by $I_q^{n}F$ the subgroup of the Witt group of quadratic forms of the field $F$ which is generated by the Witt classes of forms similar to quadratic $n$-fold Pfister forms over $F$.
Let $\mathcal{F}_{n,m}$ be the functor from the category of fields containing $k$ to the category of sets that assigns to a given field $F$ the set of isometry classes of $m$-tuples $(\varphi_1,\ldots,\varphi_m)$ of quadratic $n$-fold Pfister forms such that the Witt class of $\varphi_1 \perp \ldots \perp \varphi_m$ lies in $I_q^{n+1}F$.
The essential dimension $\operatorname{ed}(S)$ of such an $m$-tuple $S=(\varphi_1,\ldots,\varphi_m)$ is the minimal transcendence degree of a field $L$ containing $k$ and contained in $F$ such that there exist quadratic $n$-fold Pfister forms $\psi_1,\ldots,\psi_m$ over $L$ with $(\psi_i)_F \simeq \varphi_i$ for $1\leq i\leq m$ and the Witt class of $\psi_1 \perp \ldots \perp \psi_m$ lies in $I_q^{n+1}L$.
The essential dimension $\operatorname{ed}(\mathcal{F}_{n,m})$ of $\mathcal{F}_{n,m}$ is defined to be the supremum on $\operatorname{ed}(S)$, where $S=(\varphi_1,\dots,\varphi_m)$ ranges over all $m$-tuples of $n$-fold Pfister forms such that the Witt class of $\varphi_1  \perp \ldots \perp \varphi_m$ lies in $I_q^{n+1}F$  and $F$ ranges over all fields containing $k$.

For $n=2$, this coincides with the essential dimension of the set of isomorphism classes of $m$-tuples of quaternion algebras with the trivial tensor product, studied in \cite{CR15}. 
There, the essential dimension is computed to be $3$ for $m=3$ when $\operatorname{char}k\neq 2$. 
In \cite{Chap24}, it is shown that for $m=3$ the same value is attained in characteristic $2$, that is $\operatorname{ed}(\mathcal{F}_{2,3})=3$.

In this note, we show that $\operatorname{ed}(\mathcal{F}_{n,3})=n+1$ for any $n\geq 2$, and specifically, $\operatorname{ed}(S)=n+1$ if and only if $\varphi_1\perp-\varphi_2\perp\varphi_3$ is not hyperbolic where $S=(\varphi_1,\varphi_2,\varphi_3)$ such that the Witt class of $\varphi_1\perp-\varphi_2\perp\varphi_3$ lies in $I_q^{n+1}F$.
Furthermore, we prove that $\operatorname{ed}(\mathcal{F}_{2,4})\in \{4,5\}$ when $\cchar k=2$.

Our main references are \cite{EKM08} for the theory of quadratic forms, and \cite{Alb39} for the theory of central simple algebras.

\section{The essential dimension of $\mathcal{F}_{n,3}$}
Let $F$ be a field.
For $a_1,\ldots,a_n\in F^{\times}$, let $\langle a_1,\ldots,a_n\rangle_{\mathfrak{b}}$ denote the non-degenerate symmetric bilinear form 
$$F^n\times F^n\to F,\, ((x_1,\ldots,x_n),(y_1,\ldots,y_n))\mapsto\sum_{i=1}^{n}a_ix_iy_i.$$
Let $K$ be an étale quadratic $F$-algebra. Then 
\begin{equation*}
    K\simeq
\begin{cases}
    F[X]/(X^2-a)\quad \text{for $a\in F^{\times}$\quad if}\quad \cchar F\neq2, \\
    F[X]/(X^2+X+a)\quad \text{for $a\in F$\quad if}\quad \cchar F=2.
\end{cases}
\end{equation*} 
We denote by $\llangle a]]$ the norm form of $K/F$. We have that 
\begin{equation*}
    \llangle a]]=
\begin{cases}
    \langle1,-a\rangle\quad \text{if}\quad \cchar F\neq2 \\
    [1,a]\quad \text{if}\quad \cchar F=2,
\end{cases}
\end{equation*} 
where for $a,b\in F$, we denote by $[a,b]$ the $2$-dimensional nonsingular quadratic form $aX^2+XY+bY^2$.

For $n\geq 2$, the quadratic $n$-fold Pfister form is a nonsingular quadratic form isometric to
$$\langle1,-a_1\rangle_{\mathfrak{b}}\otimes\ldots\otimes\langle1,-a_{n-1}\rangle_{\mathfrak{b}}\otimes\llangle a_n]]$$
where $a_1,\ldots,a_{n-1}\in F^{\times}$ and $a_n\in F$ with $a_n\neq0$ if $\cchar F\neq2$. This form is denoted by $\llangle a_1,\ldots, a_{n-1},a_n]]$. 
A quadratic $1$-fold Pfister form is isometric to $\llangle a]]$ for some $a\in F$ with $a\neq0$ if $\cchar F\neq2$.

An $F$-quaternion algebra is a central simple $F$-algebra of degree $2$.  
For $a\in F$ and $b\in F^{\times}$, we denote by $[a,b)_F$ the $F$-quaternion algebra generated by two elements $u$ and $v$ satisfying the relations 
\[u^2-u=a,\ v^2=b\mbox{ and }vu=(1-u)v.\]
Any $F$-quaternion algebra is isomorphic to $[a,b)_F$ for certain $a\in F$ and $b\in F^{\times}$.
Let $Q$ be an $F$-quaternion algebra. We denote by $n_Q$ the reduced norm form of $Q$, which is a quadratic $2$-fold Pfister form.
It is well known that two quaternion algebras are isomorphic if and only if their reduced norm forms are isometric (see \cite[Cor. 12.5]{EKM08}).\\

Let $F$ be a field containing $k$. Let $S=(\varphi_1,\varphi_2,\varphi_3) \in \mathcal{F}_{n,3}(F)$.
By \cite[Prop. 24.5]{EKM08}, we have $\varphi_1\simeq\llangle b_1,a_1,\dots,a_{n-1}]]$, $\varphi_2\simeq\llangle b_2,a_1,\dots,a_{n-1}]]$ and $\varphi_3\simeq\llangle b_1b_2,a_1,\dots,a_{n-1}]]$ for some $b_1,b_2,a_1,\dots,a_{n-2}\in F^{\times}$ and $a_{n-1}\in F$ with $a_{n-1}\neq0$ if $\cchar F\neq2$. 
Therefore $\operatorname{ed}(S)\leq n+1$.
This implies that $\operatorname{ed}(\mathcal{F}_{n,3})\leq n+1$ as well. 
We will show that this is an equality (Corollary \ref{ed(n,3)}).

For $m, n\geq 1$ integers 
and for $S=(\varphi_1,\ldots,\varphi_m)\in\mathcal{F}_{n,m}(F)$, we let $$\Sigma_S=\varphi_1\perp-\varphi_2\perp\ldots\perp(-1)^{m-1}\varphi_m$$
and note that the Witt class of $\Sigma_S$ lies in $I_q^{n+1}F$.

\begin{thm}\label{rel-ed(S)-hyperS}
    Let $S=(\varphi_1,\varphi_2,\varphi_3) \in \mathcal{F}_{n,3}(F)$.
    Then, $\operatorname{ed}(S)\leq n$ if and only if $\Sigma_S$ is hyperbolic.
\end{thm}
\begin{proof}
    We have $\varphi_1\simeq\llangle b_1,a_1,\dots,a_{n-1}]]$, $\varphi_2\simeq\llangle b_2,a_1,\dots,a_{n-1}]]$ and $\varphi_3\simeq\llangle b_1b_2,a_1,\dots,a_{n-1}]]$ for some $b_1,b_2,a_1,\dots,a_{n-2}\in F^{\times}$ and $a_{n-1}\in F$.
    Then the Witt class of $\Sigma_S$ is given by $\llangle b_1,-b_2,a_1,\dots,a_{n-1}]]$. 
    Assume that $\operatorname{ed}(S)\leq n$. 
    Then there exists a field $L$ of transcendence degree at most $n$ over $k$ contained in $F$ and a quadratic form $\psi$ over $L$ such that 
    $\psi_F\simeq\llangle b_1,-b_2,a_1,\dots,a_{n-1}]]$.
    In particular $\dim\psi=2^{n+1}>2^n$, hence it follows by \cite[Chap. 2, Cor. 15.3]{Schar85} that $\psi$ is isotropic over $L$.
    It follows that $\Sigma_S$ is hyperbolic.

    For the other direction, assume that $\Sigma_S$ is hyperbolic. Clearly, $\operatorname{ed}(S)\leq n$ if $\phi_1$ or $\phi_2$ is isotropic.
    Assume now that these two forms are anisotropic. 
    Since $\Sigma_S$ is hyperbolic, it follows by the dimension count that its subform $-b_1 \llangle a_1,\dots,a_{n-1}]] \perp b_2 \llangle a_1,\dots,a_{n-1}]] \perp \langle 1 \rangle$ is isotropic. 
    If the quadratic form $-b_1 \llangle a_1,\dots,a_{n-1}]] \perp b_2 \llangle a_1,\dots,a_{n-1}]]$ is isotropic, then $\varphi_1 \simeq \varphi_2$, and the rest follows easily.
    Assume now that $-b_1 \llangle a_1,\dots,a_{n-1}]] \perp b_2 \llangle a_1,\dots,a_{n-1}]]$ is anisotropic.
    Then, $-b_1f+b_2g+1=0$ for some nonzero $f,g\in F$ represented by $\llangle a_1,\dots,a_{n-1}]]$.
    Set $b_1'=b_1f$ and $b_2'=b_2g$. It follows by \cite[Cor. 6.13 and Lemma 15.5]{EKM08} that $\varphi_1\simeq\llangle b_1',a_1,\dots,a_{n-1}]]$, $\varphi_2\simeq\llangle b_2',a_1,\dots,a_{n-1}]]$ and $\varphi_3\simeq\llangle b_1'b_2',a_1,\dots,a_{n-1}]]$.
    We let $L=k(a_1,\dots,a_{n-1},b_1')$ and note that $b_2'=b_1'-1\in L$. Hence, we obtain that $\operatorname{ed}(S)\leq n$.
\end{proof}

\begin{cor}\label{ed(n,3)}
    We have that $\operatorname{ed}(\mathcal{F}_{n,3})=n+1$. 
\end{cor}
\begin{proof}
    Set $F=k(X_1,\ldots,X_{n-1},Y_1,Y_2)$ where $X_1,\ldots,X_{n-1},Y_1,Y_2$ are variables over $k$. 
    Let $S=(\varphi_1,\varphi_2,\varphi_3)\in\mathcal{F}_{n,3}(F)$ where $\varphi_i=\llangle Y_i,X_1,\ldots,X_{n-1}]]$ for $1\leq i\leq2$ and $\varphi_3=\llangle Y_1Y_2,X_1,\ldots,X_{n-1}]]$.
    Then the Witt class of $\Sigma_S$ is given by $\llangle Y_1,-Y_2,X_1,\ldots,X_{n-1}]]$. 
    This form is anisotropic over $F$, by an easy valuation theoretic argument (see for example \cite{CGV17}).
    It follows by Theorem \ref{rel-ed(S)-hyperS} that $\operatorname{ed}(S)=n+1$.  
    We conclude that $\operatorname{ed}(\mathcal{F}_{n,3})=n+1$.
\end{proof}

\section{Some lower and upper bounds for $\mathcal{F}_{2,m}$}

In the case of $k$ algebraically closed of characteristic 2, \cite{BC15} provides a certain lower bound: 

\begin{prop}\label{lowerbound-ed2m}  
    Suppose that $k$ is algebraically closed and $\cchar{k}=2$. Let $F=k(X,Y_1,\ldots,Y_{m-1})$ where $X,Y_1,\ldots,Y_{m-1}$ are variables over $k$ and $m\in\mathbb{N}$. 
    Consider the quadratic $2$-fold Pfister forms $\varphi_i=\langle\!\langle Y_i,X]]$ for $1\leq i\leq m-1$ and $\varphi_m=\langle\!\langle Y_1\cdots Y_{m-1},X]]$ over $F$.
    Let $S=(\varphi_1,\ldots,\varphi_m)$. Then $\operatorname{ed}(S)=m$. 
    In particular $\operatorname{ed}(\mathcal{F}_{2,m})\geq m$.
\end{prop}
\begin{proof}
    The Witt class of $\Sigma_S$ is given by $$[1,mX]\perp Y_1[1,X]\perp\ldots\perp Y_{m-1}[1,X]\perp Y_1\cdots Y_{m-1}[1,X].$$
    In \cite[Theorem 10.2]{BC15}, it is shown that the Witt class of this later form does not descend to a field of transcendence degree lower than $m$ over $k$ (there such a form is called a canonical monomial form).
    This implies in particular that $\operatorname{ed}(S)=m$. 
    It follows further that $\operatorname{ed}(\mathcal{F}_{2,m})\geq m$. 
\end{proof}

However, \cite{BRV10, Tot19} improve this bound and also allows us to lift the requirement on the characteristic and algebraic closeness of $k$.
\begin{thm}
    Let $m\geq 7$, 
    then we have $\operatorname{ed}(\mathcal{F}_{2,m})\geq 2^m-(2m+1)m-1$.
\end{thm}
\begin{proof}
    Let $\mathcal{T}_{2m+1}$ be the functor from the category of fields containing $k$ to the category of sets that assigns to a given field $F$ containing $k$ the set of isometry classes of $(2m+1)$-dimensional regular quadratic forms over $F$ with trivial discriminant and Clifford invariant (see \cite{BFT07, BRV10} and also \cite[Proof of Theorem 3.1]{CLM25}). 
    There is a surjective functorial map from $\mathcal{F}_{2,m}(F)$ to $\mathcal{T}_{2m+1}(F)$ that maps a given sequence $S\in \mathcal{F}_{2,m}(F)$ to the unique, up to isometry, $(2m+1)$-dimensional subform $\Sigma_S'$ of $\Sigma_S$ which is Witt equivalent to $\Sigma_S\perp\langle 1\rangle$.
    Therefore, $\operatorname{ed}(\mathcal{F}_{2,m})\geq \operatorname{ed}(\mathcal{T}_{2m+1})$ (\cite[Lemma 1.9]{BF03}).
    Following \cite[\S 2]{BFT07}, an analogous result to \cite[Lemma 5.1]{BRV10} can be obtained on the connection between $\operatorname{ed}(\operatorname{Spin}_{2m+1})$ and $\operatorname{ed}(\mathcal{T}_{2m+1})$ also in characteristic $2$, and in particular, $\operatorname{ed}(\mathcal{T}_{2m+1}) \geq \operatorname{ed}(\operatorname{Spin}_{2m+1})-1$.
    By \cite{BRV10, Tot19}, we have
    $\operatorname{ed}(\operatorname{Spin}_{2m+1}) = 2^{2m}-(2m+1)m$ (note that there is a typo in the statement of \cite[Theorem 2.1]{Tot19}, as pointed out in \cite{CLM25}). 
    Now, the statement follows.
\end{proof}

\begin{rem}
    In \cite{Ros99, Tot19}, lower bounds for $\operatorname{ed}(\operatorname{Spin}_{7})$ and $\operatorname{ed}(\operatorname{Spin}_{9})$ are also provided. These bounds give rise to the same lower bounds for $\operatorname{ed}(\mathcal{F}_{2,3})$ and $\operatorname{ed}(\mathcal{F}_{2,4})$ as in Proposition \ref{lowerbound-ed2m}, without assuming $k$ is algebraically closed.
    Furthermore in \cite{Ros99}, lower bounds for $\operatorname{ed}(\operatorname{Spin}_{11})$ and $\operatorname{ed}(\operatorname{Spin}_{12})$ in characteristic different from $2$ are provided. These bounds give rise to less sharp good bounds for $\operatorname{ed}(\mathcal{F}_{2,5})$ and $\operatorname{ed}(\mathcal{F}_{2,6})$ obtained in Proposition \ref{lowerbound-ed2m}.
\end{rem}

\begin{rem}
    There was a typo in copying the value of $\operatorname{ed}(\operatorname{Spin}_{2r+1})$ in \cite[Theorem 3.1]{CLM25}. The statement of that theorem should be that the symbol length of classes in $H_2^3(F)$ of $2(r+1)$-dimensional forms in $I_q^3 F$ for $r \geq 7$ is at most $\binom{2^r-(2r+1)r}{2}$.
\end{rem}

\begin{rem} In the case of $2$-fold Pfister forms, the characteristic zero analogue was studied by Cernele ad Reichstein,
and the essential dimension was determined for all cases except the case of quadruples of quaternion algebras.
For triples of quaternion algebras, the essential dimension is $3$, as is it in our case,
and for $m>4$, the essential dimension of $m$-tuples of quaternion algebras with trivial tensor product is
$2^m-3m-1$ by their result \cite[Theorem 1.3 (b)]{CR15}.
The bounds that we provide here are not as good as \cite{CR15}, but they cover the case of fields of characteristic 2. We also address $\operatorname{ed}(\mathcal{F}_{2,4})$, not covered in \cite{CR15}.
\end{rem}

Let $m\in\mathbb{N}$. Let $\mathcal{G}_m$ be the functor from the category of fields containing $k$ to the category of sets that assigns to a given field $F$ containing $k$ the set of isomorphism classes of $m$-tuples $(Q_1,\dots,Q_m)$ of $F$-quaternion algebras such that the tensor product $Q_1 \otimes_F \ldots \otimes_F Q_m$ is split over $F$.
The essential dimension $\operatorname{ed}(S)$ of such an $m$-tuple $S=(Q_1,\ldots,Q_m)$ is the minimal transcendence degree of a field $L$ containing $k$ and contained in $F$ such that there exist quaternion algebras $H_1,\dots,H_m$ over $L$ with $H_i \otimes_L F\simeq Q_i$ for $1\leq i\leq m$ and $H_1 \otimes_L \ldots \otimes_L H_m$ is split over $L$.
The essential dimension $\operatorname{ed}(\mathcal{G}_m)$ of $\mathcal{G}_m$ is defined to be the supremum on  $\operatorname{ed}(S)$, where $S=(Q_1,\dots,Q_m)$ ranges over all $m$-tuples of $F$-quaternion algebras with $Q_1 \otimes _F\ldots \otimes_F Q_m$ is split and $F$ ranges over all fields containing $k$.

Consider the natural transformation $\mathcal{N}: \mathcal{F}_{2,m}\to\mathcal{G}_m$ defined as follows. 
For a field $F$ containing $k$ and $(Q_1,\ldots,Q_m)\in\mathcal{F}_{2,m}(F)$, let $\mathcal{N}(Q_1,\ldots,Q_m)=(n_{Q_1},\ldots,n_{Q_m})$.
Then $\mathcal{N}$ defines a natural equivalence. 
In particular, we have that $\operatorname{ed}(\mathcal{F}_{2,m})=\operatorname{ed}(\mathcal{G}_m)$.
Therefore, when computing the essential dimension of $\mathcal{F}_{2,m}$, we can consider the functor $\mathcal{G}_{m}$, when it is more convenient.

\begin{prop}
    Assume $\cchar{k}=2$. Then, we have that $\operatorname{ed}(\mathcal{F}_{2,m})\leq2^m+m$.
\end{prop}
\begin{proof}
    Let $F$ be a field containing $k$ and let $S=(Q_1,\dots,Q_m)\in\mathcal{G}_m(F)$.
    By \cite[Theorem 2.1]{CQ24}, there exist $a_1,\dots,a_{m+1}\in F^{\times}$ such that $Q_i\simeq[a_i,a_{i+1})_F$ for $1\leq i\leq m$. 
    Let $L=k(a_1,\ldots,a_{m+1})$. 
    By \cite[Lemma 7.13]{Alb39}, the $L$-algebra 
    $C=\bigotimes_{i=1}^m[a_i,a_{i+1})_L$ is cyclic.
    Therefore, $C\simeq[L'/L,\sigma,b)$ for a cyclic extension $L'/L$ with $\sigma$ a generator of its Galois group and some $b\in L^{\times}$ (see \cite[Theorem 5.9]{Alb39} also for the notation). 
    Since $C\otimes_LF$ is split, we obtain by \cite[Theorem 5.14]{Alb39} that $b=\N_{FL'/F}(x)$ for some $x\in FL'$, where $\N_{FL'/F}$ denotes the norm form with respect to $FL'/F$. 
    We have $x=(x_0,\ldots,x_{2^m-1})$ with $x_i\in F$ for $0\leq i\leq 2^m-1$. 
    Let $M=L(x_1,\ldots,x_{2^m-1})$. As $b\in L$, we have that $x_0^{2^m}\in M$. 
    It follows that $C\otimes_L{M(x_0)}$ is split and that $\operatorname{trdeg}_kM(x_0)\leq2^m+m$.
    Hence $\operatorname{ed}(S)\leq2^m+m$. This argument implies that $\operatorname{ed}(\mathcal{F}_{2,m})\leq2^m+m$.   
\end{proof}

In arbitrary characteristic (including $\cchar k\neq 0$ and $\cchar k=2$, not addressed in \cite{CR15}), one can obtain the following less good upper bound for $\operatorname{ed}(\mathcal{F}_{2,m})$.
\begin{rem}
    In arbitrary characteristic, we have $\operatorname{ed}(\mathcal{F}_{2,m})\leq2^m+2m-1$. 
    Indeed, let $F$ be a field containing $k$ and let $S=(Q_1,\dots,Q_m)\in\mathcal{G}_m(F)$.
    We can clearly find a field $K\subseteq F$ of the transcendence degree at most $2m$ over $k$ and $K$-quaternion algebras $H_1,\ldots,H_m$ such that $H_i\otimes_KF\simeq Q_i$ for $1\leq i\leq m$.
    Considering the function field of the Severi-Brauer variety of the tensor product $\bigotimes_{i=1}^mH_i$, we can find a field $K\subseteq L\subseteq F$ of the transcendence degree $\leq 2^m-1$ over $K$ such that $\bigotimes_{i=1}^m(H_i\otimes_KL)$ is split.
    Now $k\subseteq L\subseteq F$ has the transcendence degree at most $2^m+2m-1$ over $k$. This yields the desired inequality.
\end{rem}

We conclude this section with the following known properties for quaternion algebras.
\begin{prop}\label{known-properties}
    Let $F$ be a field of characteristic $2$.
    Let $a,a_1,a_2\in F$ and $b,b_1,b_2\in F^{\times}$. The following hold:
\begin{enumerate}[$(1)$]
    \item $[a,b)_F$ is split if and only if $a=\lambda^2+\lambda+\mu^2b$ for some $\lambda,\mu\in F$.
    \item $[a_1,b_1)_F\otimes_F[a_2,b_2)_F\simeq[a_1,b_1b_2)_F\otimes_F[a_1+a_2,b_2)_F$.
\end{enumerate}    
\end{prop}
\begin{proof}
    $(1)$ By \cite[p.104]{Draxl83}, the quaternion algebra $[a,b)_F$ is split if and only if $b=x^2+xy+ay^2$ for some $x,y\in F$. 
    Hence, it suffices to show that this later equality is equivalent to the one in the statement.
    For such $x,y$ with $y\neq0$, we let $\lambda=xy^{-1}$ and $\mu=y^{-1}$. 
    For $x,y$ with $y=0$, we have $b=x^2$ and in particular $x\neq0$, and in this case, we let $\lambda=a$ and $\mu=ax^{-1}$. 
    Then, $a=\lambda^2+\lambda+\mu^2 b$.
    
    $(2)$ This follows by the rules for quaternion algebras given in \cite[p.104]{Draxl83}. 
\end{proof}

\section{The essential dimension of $\mathcal{F}_{2,4}$ in characteristic 2}
In this section, we compute the essential dimension of $\mathcal{F}_{2,4}$ when $\cchar k=2$.
This case is not discussed in \cite{CR15}.

For a field $F$ of characteristic $2$ and a separable quadratic extension $F(\wp^{-1}(a))$ of $F$ with $a\in F$, we denote by $\N_a$ the norm map $F(\wp^{-1}(a))\to F$.

\begin{lem}\label{linkage-rel-slots}
    Let $F$ be a field of characteristic $2$.
    Let $a_1,a_2,a_3\in F$ and $b_1,b_2,b_3\in F^{\times}$. Assume $\bigotimes_{i=1}^{3}[a_i,b_i)_F$ is split. 
    Then, there exist $\alpha_i\in F(\wp^{-1}(a_i))$ for $1\leq i\leq3$ such that 
    $$a_1+b_1\N_{a_1}(\alpha_1)\equiv a_2+b_2\N_{a_2}(\alpha_2)\equiv a_3+b_3\N_{a_3}(\alpha_3)\mod{\wp(F)}.$$
\end{lem}
\begin{proof}
    By \cite[\S 14, Cor. 1]{Draxl83}, the three quaternion algebras have a common left slot. 
    The possible left slots of a quaternion algebra $[a,b)_F$ are the reduced norms of elements whose reduced trace is $1$, namely the expressions $a+\lambda^2+\lambda+b \N_a(\alpha)$ for some $\lambda \in F$ and $\alpha \in F(\wp^{-1}(a))$. The statement readily follows.
\end{proof}

Let $Q$ be an $F$-quaternion algebra. For $a\in F$, $b\in F^{\times}$ with $Q\simeq[a,b)_F$, we have that $n_Q\simeq[1,a]\perp b[1,a]$. 

\begin{thm}
    When $\cchar k=2$,  we have that $\operatorname{ed}(\mathcal{F}_{2,4})\in\{4,5\}$.
\end{thm}
\begin{proof}
    By Proposition \ref{lowerbound-ed2m}, we have $\operatorname{ed}(\mathcal{F}_{2,4})\geq4$. Hence we are left to show that the essential dimension is bounded above by $5$.

    Let $F$ be a field containing $k$.
    Let $Q_1,\ldots, Q_4$ be $F$-quaternion algebras such that $\bigotimes_{i=1}^4Q_i$ is split, and let
    $S=(Q_1,\ldots,Q_4)$.
    We may assume without loss of generality that $Q_1,\ldots,Q_4$ are all division algebras.
    Let $a_1,a_2,a_3,a_4\in F$ and $b_1,b_2,b_3,b_4\in F^{\times}$ be such that $Q_i\simeq[a_i,b_i)_F$ for $1\leq i\leq4$.
    The Witt class of $n_{Q_1}\perp\ldots\perp n_{Q_4}$ is given by
    $$[1,a_1+a_2+a_3+a_4] \perp b_1[1,a_1]\perp b_2[1,a_2]\perp b_3[1,a_3]\perp b_4[1,a_4]$$
    which then lies in $I_q^3F$, as $\bigotimes_{i=1}^4Q_i$ is split. 
    Note that the dimension of this latter form is $10$, hence by \cite[Theorem 4.10]{DM02} it is isotropic.
    By \cite[Lemma 3.1]{CM23}, there exist some $\lambda\in F$ and $\alpha_i\in F(\wp^{-1}(a_i))$ for $1\leq i\leq4$ such that 
    $$\lambda^2+\lambda+a_1+a_2+a_3+a_4+\sum_{i=1}^4 b_i\N_{a_i}(\alpha_i)=0.$$ 
    Set $d_i=b_i\N_{a_i}(\alpha_i)$ and $c_i=a_i+d_i$ when $\alpha_i\neq 0$ and $d_i=b_i$ and $c_i=a_i$ when $\alpha_i=0$.
    Note that $$c_1+c_2+c_3+c_4 \equiv 0 \mod{\wp{(F)}}.$$
    Hence $Q_i\simeq[c_i,d_i)_F$ for $1\leq i\leq3$ and $Q_4\simeq[c_1+c_2+c_3,d_4)_F$.
    Applying Proposition \ref{known-properties} (2) to the pairs $(Q_1,Q_4),(Q_2,Q_4),(Q_3,Q_4)$, we obtain that the $F$-algebra $\bigotimes_{i=1}^3[c_i,d_id_4)_F$ is split. 
    We continue the proof by treating possible cases according to whether these three quaternion algebras are division or not. 
    There are three cases; all three being split, all three being division, and only one of them being split.

    Consider first the case where each of the quaternion algebras $[c_i,d_id_4)_F$ is split. It follows that $Q_i\simeq[c_i,d_4)_F$ for $1\leq i\leq3$ and $Q_4\simeq[c_1+c_2+c_3,d_4)_F$. 
    Let $L=k(c_1,c_2,c_3,d_4)$. Then $L/k$ is of transcendence degree at most $4$ and we have that $\bigotimes_{i=1}^3[c_i,d_4)_L\otimes_L[c_1+c_2+c_3,d_4)_L\sim[0,d_4)_L$, which is split.

    Assume now that each of the quaternion algebra $[c_i,d_id_4)_F$ is division.
    By Lemma \ref{linkage-rel-slots}, we obtain that
    $$c_1+d_1d_4\N_{c_1}(\beta_1) \equiv c_2+d_2d_4\N_{c_2}(\beta_2) \equiv c_3+d_3d_4\N_{c_3}(\beta_3) \mod{\wp(F)}$$
    for some $\beta_i\in F(\wp^{-1}(c_i))$ with $1\leq i\leq3$.
    Let $f_4=d_4$, and for $1\leq i\leq3$, let $f_i=d_i\N_{c_i}(\beta_i)$, $\delta_i=1$ if $\beta_i\neq0$ and $f_i=d_i$, $\delta_i=0$ if $\beta_i=0$. Then we have that
    $$c_1+\delta_1f_1f_4\equiv c_2+\delta_2f_2f_4\equiv c_3+\delta_3f_3f_4 \mod{\wp(F)}$$
    and that $[c_i,d_id_4)_F\simeq[c_i+\delta_if_if_4,f_if_4)_F$ for $1\leq i\leq3$. 
    It follows that the quaternion algebra $[c_1+\delta_1f_1f_4,f_1f_2f_3f_4)_F$ is split.
    By Proposition \ref{known-properties}, there exist some $x\in F$ such that $$c_1+\delta_1f_1f_4\equiv x^2f_1f_2f_3f_4\mod{\wp(F)}.$$ 
    Note that $x\neq0$, since $[c_1,d_1d_4)_F$ is a division algebra.
    Now we can find some $u_1,u_2,u_3\in F$ such that 
    $$c_1+u_1^2+u_1+\delta_1f_1f_4=c_2+u_2^2+u_2+\delta_2f_2f_4=c_3+u_3^2+u_3+\delta_3f_3f_4=x^2f_1f_2f_3f_4.$$
    We set $e_i=c_i+u_i^2+u_i$ for $1\leq i\leq3$. Then, we have $Q_i\simeq[e_i,f_i)_F$ for $1\leq i\leq3$ and $Q_4\simeq[e_1+e_2+e_3,f_4)_F$.  
    Let $L=k(e_1,f_1,f_2,f_3,f_4)(x)$. Then $\operatorname{trdeg}_kL\leq5$.
    Furthermore, using Proposition \ref{known-properties}, one can easily check that the $L$-algebra $\bigotimes_{i=1}^3[e_i,f_i)_L\otimes_L[e_1+e_2+e_3,f_4)_L$ is split. 

    Finally assume that only one of quaternion algebras $[c_i,d_id_4)_F$, say $[c_1,d_1d_4)_F$, is split.
    Then $Q_1\simeq[c_1,d_4)_F$ and the $F$-algebra $[c_2,d_2d_4)_F\otimes_F[c_3,d_3d_4)_F$ is split.
    Arguing as in the proof of Lemma \ref{linkage-rel-slots}, we can find some $\beta_2\in F(\wp^{-1}(c_2))$ and $\beta_3\in F(\wp^{-1}(c_3))$ such that
    $$c_2+d_2d_4\N_{c_2}(\beta_2)\equiv c_3+d_3d_4\N_{c_3}(\beta_3)\mod{\wp(F)}.$$
    Let $f_1=d_1, f_4=d_4$, and let $f_i=d_i\N_{c_i}(\beta_i)$, $\delta_i=1$ if $\beta_i\neq0$ and $f_i=d_i$, $\delta_i=0$ if $\beta_i=0$, for $2\leq i\leq3$. 
    Then the $F$-quaternion algebra $[c_2+\delta_2f_2f_4,f_2f_3)_F$ is split.
    By Proposition \ref{known-properties} there exists some $x\in F$ such that $$c_2+\delta_2f_2f_4\equiv x^2f_2f_3\mod{\wp(F)}.$$ 
    Note that $x\neq0$, since $[c_2,d_2d_4)_F$ is a division algebra.
    Now we can find some $u_2,u_3\in F$ such that 
    $$c_2+u_2^2+u_2+\delta_2f_2f_4=c_3+u_3^2+u_3+\delta_3f_3f_4=x^2f_2f_3.$$
    We set $e_1=c_1$, $e_i=c_i+u_i^2+u_i$ for $2\leq i\leq3$. Then we have $Q_1\simeq[e_1,f_4)_F$, $Q_i\simeq[e_i,f_i)_F$ for $2\leq i\leq3$ and $Q_4\simeq[e_1+e_2+e_3,f_4)_F$. 
    Let $L=k(e_1,e_2,f_2,f_3,f_4)(x)$.
    Then $\operatorname{trdeg}_kL\leq5$.
    Furthermore, using Proposition \ref{known-properties}, one can easily check the $L$-algebra $\bigotimes_{i=2}^3[e_i,f_i)_L\otimes[e_1,f_4)_L\otimes_L[e_1+e_2+e_3,f_4)_L$ is split. 

    We conclude that $\operatorname{ed}(S)\leq5$. 
    This shows that $\operatorname{ed}(\mathcal{F}_{2,4})\leq5$ and completes the proof.
\end{proof}

\begin{exmp}\label{linkage-S_3}
    Let $F$ be a field containing $k$ of $\cchar{k}=2$ and let $S\in\mathcal{G}_{3}(F)$. Then $S$ is a separably linked sequence.
    It follows by \cite[Example 5.4]{CGV17} that the form $\Sigma_S$ is hyperbolic if and only if the sequence $S$ is inseparably linked.
\end{exmp}

\begin{exmp}
    Let $F$ be a field containing $k$ of $\cchar{k}=2$ and let $S=(Q_1,Q_2,Q_3,Q_4)\in\mathcal{G}_{4}(F)$ be a separably linked sequence.
    Let $H_i$ be the $F$-quaternion algebra Brauer equivalent to $Q_1\otimes_FQ_{i+1}$ for $1\leq i\leq3$.
    Then $\Sigma_S$ is hyperbolic if and only if $(H_1,H_2,H_3)$ is inseparably linked.

    Indeed, let $a\in F$, $b_1,b_2,b_3\in F^{\times}$ be such that $Q_i\simeq[a,b_i)_F$ for $1\leq i\leq 3$ and $Q_4\simeq[a,b_1b_2b_3)_F$.
    Then the Witt class of $\Sigma_S$ is given by $b_1 \llangle b_1b_2,b_1b_3,a]]$ and $H_i\simeq[a,b_1b_{i+1})_F$ for $1\leq i\leq3$.
    Let $S'=(H_1,H_2,H_3)$. Then $S'$ is separably linked and the Witt class of $\Sigma_{S'}$ is given by $\llangle b_1b_2,b_1b_3,a]]$. 
    Since $\Sigma_S=b_1\Sigma_{S'}$,
    it follows by Example \ref{linkage-S_3} that $S'$ is inseparably linked if and only if $\Sigma_S$ is hyperbolic.
\end{exmp}

\subsection*{Acknowledgement}
This work was supported by the project \emph{IEA of CNRS} between the Artois University and the Academic College of Tel-Aviv-Yaffo, by the 2020 PRIN (project \emph{Derived and underived algebraic stacks and applications}) from MIUR, and by research funds from Scuola Normale Superiore.

\end{document}